\documentclass[12pt, reqno]{amsart}
\usepackage{amscd}
\usepackage{amsmath}
\usepackage{amssymb}
\usepackage{amsthm}
\usepackage{amsxtra,hyperref}
\usepackage{array}
\usepackage{enumitem}
\usepackage[dvips]{graphicx}
\usepackage{latexsym}

\usepackage[mathscr]{eucal}
\usepackage{eucal} 

\usepackage{graphicx,psfrag,color}
\usepackage{epsfig,pstricks}
 \usepackage{mathrsfs}
\usepackage{calligra}
\usepackage[T1]{fontenc}

\theoremstyle{plain}
\newtheorem{theorem}{Theorem}[section]
\newtheorem{lemma}[theorem]{Lemma}
\newtheorem{proposition}[theorem]{Proposition}

\theoremstyle{definition}
\newtheorem{remark}[theorem]{Remark}

\allowdisplaybreaks

\newtheorem*{SP}{Spectral~Property} 

\definecolor{orange}{rgb}{0.995, 0.75, 0.35}
\definecolor{purple}{rgb}{0.7, 0.2, 0.5}
\definecolor{royalblue}{rgb}{0.2, 0.7, 0.8}
\definecolor{darkgreen}{rgb}{0.2,0.725,0.25} 

\newcommand{\norm}[1]{ \left|  #1 \right|}

\def\al{\alpha}
\def\de{\delta}

\def\ga{\gamma}
\def\veps{\varepsilon}
\def\vphi{\varphi}

\def\De{\Delta}

\def\Lam{\Lambda}
\def\Om{\Omega}

\def\dive{\mathrm{div}\,}

\def\iy{\infty}

\def\pa{\partial}

\def\sH{\mathscr{H}}

\def\trace{\mathrm{tr}\;}

\newcommand{\inv}{^{-1}}

\newcommand{\nd}{\noindent}

\newcommand{\R}{\mathbb{R}}
\newcommand{\C}{\mathbb{C}}

\def\sideremark#1{\ifvmode\leavevmode\fi\vadjust{\vbox to0pt{\vss
 \hbox to 0pt{\hskip\hsize\hskip1em
\vbox{\hsize2cm\tiny\raggedright\pretolerance10000
 \noindent #1\hfill}\hss}\vbox to8pt{\vfil}\vss}}}%

{\end{list}}

\begin{document}
\title[Blowup rate for rotational NLS]{Blowup rate for mass critical rotational nonlinear Schr\"odinger equations}
\author{Nyla Basharat}
\address[Nyla Basharat]{
DPHS\\
Augusta University\\
Augusta, GA 30912}
\email{nbasharat@augusta.edu}
\author{Yi Hu}
\address[Yi Hu]{Department of Mathematical Sciences \\
Georgia Southern University \\
Statesboro, GA 30460}
\email{yihu@georgiasouthern.edu}
\author{Shijun Zheng}
\address[Shijun Zheng]{DMS\\
Georgia Southern University \\
Statesboro, GA 30460}
\email{szheng@georgiasouthern.edu}

\date{\today}
\subjclass[2010]{Primary  35Q55, 35B44. Secondary 35P30}
\keywords{blowup rate, harmonic potential, angular momentum}

\begin{abstract}
We consider the blowup rate for blowup  solutions to $L^2$-critical, focusing NLS 
with a harmonic potential and a rotation term.
Under a suitable spectral condition we  prove that there holds the ``$\log$-$\log$ law'' when the initial data is slightly above the ground state. 
We also construct  minimal mass blowup solutions near the ground state level with distinct blowup rates. 
\end{abstract}

\maketitle


\section{Introduction}
 Consider the focusing nonlinear Schr\"odinger equation (NLS) with an angular momentum term on $\mathbb{R}^{1+n}$:
	\begin{align}\label{eq:nls-u}
	\begin{cases}
	i u_t = - \Delta u + V u - |u|^{p-1} u + i A \cdot \nabla u \\
	u(0, x)=u_0\in \sH^1,
	\end{cases}
	\end{align}
where $u = u(t, x)$: $\mathbb{R} \times \mathbb{R}^n\to\C$ denotes the wave function, 
$V(x)= \gamma^2 |x|^2$, $\gamma>0$ is a trapping harmonic potential that confines the movement of particles,  
and $A(x) = M x$, 
with $M=-M^T$ being an $n \times n$ real-valued skew-symmetric matrix.  
The linear hamiltonian $H_{A, V} := -\Delta + V + i A \cdot \nabla$ is essentially self-adjoint in $L^2$,
whose eigenvalues are associated to the Landau levels as quantum numbers. 
The angular momentum operator $L_A u:=i A\cdot\nabla u$  generates the rotation $e^{tA\cdot \nabla }u=u(e^{tM}x)$ in $\mathbb{R}^n$. 
The space $\sH^1=\sH^{1,2}$ denotes the weighted Sobolev space given by: for $r\in (1,\iy)$,
	\begin{align*}	\mathscr{H}^{1,r}(\mathbb{R}^n)
	:= \left\{ f \in L^r:  \nabla f,\,  x f \in L^r\right\},
	\end{align*}
which is endowed with the norm $\norm{f}_{\sH^{1,r}}=\norm{\nabla f}_r+\norm{x f}_r+\norm{ f}_r$, here  
$\norm{\cdot}_r:=\norm{\cdot}_{L^r}$ denoting the $L^r$-norm. 

When $n =3$, equation \eqref{eq:nls-u} is also known as  Gross-Pitaevskii equation which models rotating Bose-Einstein condensation (BEC) with attractive particle interactions in a dilute gaseous ultra-cold superfluid.  
The operator $L_A$ are usually denoted by $- \Omega \cdot L $,
where $\Omega=(\Om_1,\Om_2,\Om_3) \in \mathbb{R}^3$ is a given angular velocity vector and $L := -i x \wedge \nabla$, in which case
$M= \begin{pmatrix} 0& -\Om_3& \Om_2\\
\Om_3&0& -\Om_1\\   
-\Om_2&\Om_1&0
\end{pmatrix}$.  

Such a system describing rotating particles in a 
harmonic trap has acquired significance in connection with optics and atomic physics 
in theoretical and experimental physics \cite{Gross61,MAHHWC99,ReZaStri01,
BaoWMar05,Af06,BaoCai15}. 
Meanwhile, it demands rigorous mathematical analysis on the evolution and dynamics of the quantized flow. 
For $p\in (1, 1+4/(n-2))$, the local in time existence and uniqueness of (\ref{eq:nls-u}) has been established in  \cite{HaoHsiaoLi1,HaoHsiaoLi2,AnMaSpar,BaHaHuZheng},
see also \cite{CazE88,De,Zheng} for the treatment in a general magnetic setting.  
The purpose of this article is to study how the rotation  affects the  wave collapse as well as energy concentration under a  trapping potential.
In particular, we will address the blowup rate for the blowup solution of the $L^2$-critical focusing equation in \eqref{eq:nls-u} where $p=1+4/n$. 

Let $H^1=\{u\in L^2:\nabla u\in L^2\}$ be the usual Sobolev space. Recall that the standard NLS reads 
	\begin{align}\label{eq:nls-phi}
	\begin{cases}
	i \varphi_t = - \Delta \varphi - |\varphi|^{p-1} \varphi, \\
	\varphi(0, x) = \varphi_0\in H^1,
	\end{cases}
	\end{align}
and the well-posedness problem for $p\in (1,1+4/(n-2)]$ has been studied for a few decades and is quite well understood in the euclidean space.
Let $Q \in H^1(\mathbb{R}^n)$ be the unique   positive radial function that satisfies (\cite{Kw89,Wein83})
	\begin{align}\label{eq:ground-state}
	-Q = -\Delta Q - Q^{1+\frac{4}{n}}.
	\end{align}
When $p=1 + \frac{4}{n}$  and $\norm{\vphi_0}_2=\norm{Q}_2$,
Merle \cite{Mer} was able to determine the profile for all blowup solutions with minimal mass at the ground state level, which are obtained from pseudo-conformal transform. 
Hence
 all blowup solutions have blowup rate $(T-t)\inv$. 
In the mass critical and supercritical case $p\in [1+4/n,\iy)$, the wave collapse dynamics appears very subtle issue for $\norm{u_0}_2=\norm{Q}_2$.  
Within an arbitrarily small neighborhood of $Q$, there always exist $\vphi_0$ and $\psi_0$ such that the flow 
$\vphi_0\mapsto \vphi$ blowups in finite time, and, $\psi_0\mapsto \psi$ exists global in time and scatters as $t\to \iy$ in $H^1$. 

When $p=1+\frac{4}{n}$ and $\norm{\vphi_0}_2>\norm{Q}_2$, Bourgain and Wang \cite{BW98} 
constructed solutions of positive energy having blowup rate $(T-t)\inv$ in dimensions $n=1, 2$, 
which  was later shown unstable however.  
Perelman \cite{Pere01}  gave the first rigorous demonstration of the existence and stability for the $\log$-$\log$ speed for generic blowup solutions in 1d. 
More recently, under certain spectral condition in the Spectral Property (Section \ref{s:spec-loglog}), 
Merle and Rapha\"el \cite{MerRa1} 
proved the sharp blowup rate of the solutions for  \eqref{eq:nls-phi},
i.e., there exists a universal constant $\al^*$ such that 
if $\varphi_0\in B_{\al^*}$ with negative energy,
then  $\vphi(t,x)$ is a blowup solution to \eqref{eq:nls-phi} with maximal interval of existence $[0, T)$
 satisfying the $\log$-$\log$ blowup rate as $t\to T$
	\begin{align*}
         \vert \nabla \varphi(t) \vert_{2}
	\approx\sqrt{ \frac{\log |\log(T - t)|}{T-t} },
	\end{align*}
where
	\begin{align*}
	  B_{\alpha}
	:= \left\{ \phi \in H^1: \int |Q|^2 dx < \int |\phi|^2 dx < \int |Q|^2 dx + \alpha \right\},
	\end{align*}
see Theorem \ref{thm:log-log-phi}.
Such a blowup rate is also known to be stable in $H^1$.  

In the presence of a rotational term,  we will show how to prove   such a ``$\log$-$\log$ law''  
for the NLS  \eqref{eq:nls-u}.   
Our proof is based on a virial identity for \eqref{eq:nls-u}, 
 the $\mathcal{R}$-transform \eqref{eq:phi-to-u} that maps solutions of \eqref{eq:nls-phi} to solutions of \eqref{eq:nls-u}, and an application of the above result of Merle and Rapha\"el's. 
 This treatment is motivated by  \cite{ZhuZhang}, where the analogous result is obtained for the case $A=0$ and $V$ being a harmonic potential.  
Our main result is stated as follows. Let $\al^*$ be the above-mentioned universal constant. 
\begin{theorem}\label{thm:log-log-u}
Let $p = 1 + \frac{4}{n}$ in \eqref{eq:nls-u},  $1\le n \leq 5$. 
Suppose $u_0\in B_{\al^*}$ satifies
	\begin{align}\label{eq:negative-initial-energy}
	\int |\nabla u_0|^2 -\frac{n}{n+2} \int |u_0|^{2+\frac{4}{n}}
	< 0.
	\end{align}
Then $u\in C([0, T); \mathscr{H}^1)$ is a blowup solution of  \eqref{eq:nls-u} in finite time $T < \infty$, 
with the $\log$-$\log$ blowup rate
	\begin{align*}
	\lim_{t \rightarrow T}
	\frac{\vert \nabla u(t) \vert_2}{\vert \nabla Q \vert_{2}}
	\sqrt{ \frac{T - t}{\log \left| \log (T - t)  \right|} }
	= \frac{1}{\sqrt{2 \pi}},
	\end{align*}
where $Q$ is the unique solution of \eqref{eq:ground-state}.
\end{theorem}

\section{Local wellposedness for \eqref{eq:nls-u}} 
For $\varphi_0 \in \mathscr{H}^1$, 
the local well-posedness of \eqref{eq:nls-u} was proved for $1 \leq p < 1 + \frac{4}{n-2}$,  
see e.g. in \cite{De, Zheng}.  %
  The proof for the local well-posedness relies  on local in time dispersive estimates for $U(t)=e^{-itH_{A,V}}$,  
the fundamental solution on $[0,\de)$ (for some small $\de>0$) constructed in \cite{Ya}. 
The vectorial function $A$  represents a magnetic potential that induces the Coriolis effect or centrifugal force for the spinor particles. 
Alternatively, 
this can also be done by means of the  explicit formula in \eqref{e:fund-U(t)} for $U(t)$ defined in $(0,\pi/2\ga)$. 
This formula is obtained from the so-called $\mathcal{R}$-transform in \eqref{eq:phi-to-u}, a type of pseudo-conformal transform.

\begin{proposition}\label{p:blup-altern-conserv} 
Let $1 + \frac{4}{n} \leq p < 1 + \frac{4}{n-2}$. Suppose $u_0 \in \mathscr{H}^1$.  
\begin{enumerate}
\item[\textup{(a)}]\;\textup{[Blowup Alternative]}  
	\begin{enumerate}[label=\rm{(\alph*)}]
	\item[\textup{(i)}] If $1 \leq p < 1 + \frac{4}{n-2}$,  then there exists $T^* > 0$ such that \eqref{eq:nls-u} has a unique maximal solution $u \in C([0, T^*), \mathscr{H}^1)\cap L^q_{loc}([0, T^*), \mathscr{H}^{1,r})$, where $r = p+1$  and $q= \frac{4(p+1)}{n(p-1)}$.
\item[\textup{(ii)}] If $T^*$ is finite, then $\norm{\nabla u}_2\to \iy$  as $t\to T^*$ with a lower bound:
\begin{align*}
&\norm{\nabla u(t)}_2\ge \frac{C}{\sqrt{T^*-t}}.
\end{align*}
\end{enumerate}	
	\item[\textup{(b)}]\;\textup{[Conservation Laws]} The following are conserved on the maximal lifespan $[0,T^*)$.
		\begin{enumerate}[label=\rm{(\roman*)}]
		\item[]\;\textup{(mass)}
		$\displaystyle M(u)=\int |u|^2$
	\item[]\;\textup{(energy)}
		$\displaystyle E(u)=\int \left( |\nabla u|^2 +V |u|^2 - \frac{2}{p+1} |u|^{p+1} + i \bar{u} A\cdot \nabla u \right)$
	\item[]\;\textup{(angular momentum)}
		$\displaystyle \ell_A(u)=\int i \bar{u} A\cdot \nabla u$.
		\end{enumerate}
\end{enumerate} 
\end{proposition} 
In the critical case $p=1+4/n$, from \cite{BaHaHuZheng} we know that  $\norm{Q}_2$ is the sharp threshold such that: 
\begin{enumerate}
\item[(a)] If $\norm{u_0}<\norm{Q}_2$, then \eqref{eq:nls-u} has a unique global in time solution.
\item[(b)] For all $c\ge \norm{Q}_2$, there exists $u_0$ with   $\norm{u_0}_2=c$ so that $u$ is a finite time blowup solution of \eqref{eq:nls-u}.
\end{enumerate}
As we mentioned in the introduction section, if $\norm{u_0}=\norm{Q}_2$, 
 from Merle's characterization for the blowup profile of \eqref{eq:nls-phi}, all such blowup solutions have the blowup rate 
 $\norm{\nabla u(t)}_2\approx \frac{C}{T^*-t}$ as $t\to T^*$, see Proposition  \ref{c:charact-mini-pc}.

\section{A spectral property and the $\log$-$\log$ law}\label{s:spec-loglog} 
To  show the blowup rate for initial data above the ground state $Q$ 
as stated in  Theorem \ref{thm:log-log-u},
 we need the following Spectral Property.
Let  $y$ denote the spatial variable in $\R^n$.
\begin{SP}\label{h:spectral_property}
Consider the two  Schr\"odinger operators
	\begin{align*}
	L_1
	:= - \Delta + \frac{2}{n} \left( \frac{4}{n} + 1 \right) Q^{\frac{4}{n} - 1} y \cdot \nabla Q, \qquad
	L_2
	:= - \Delta + \frac{2}{n} Q^{\frac{4}{n} - 1} y \cdot \nabla Q,
	\end{align*}
and the real-valued quadratic form for $\varepsilon = \varepsilon_1 + i \varepsilon_2 \in H^1$ 
	\begin{align*}
	H(\varepsilon, \varepsilon)
	:= (L_1 \varepsilon_1, \varepsilon_1) + (L_2 \varepsilon_2, \varepsilon_2).
	\end{align*}
Let
	\begin{align*}
	Q_1
	:= \frac{n}{2} Q + y \cdot \nabla Q, \qquad
	Q_2
	:= \frac{n}{2} Q_1 + y \cdot \nabla Q_1.
	\end{align*}
Then there exists a universal constant $\delta_0 > 0$ such that for every $\varepsilon \in H^1$,
if
	\begin{align*}
	(\varepsilon_1, Q)
	= (\varepsilon_1, Q_1)
	= (\varepsilon_1, y_j Q)_{1 \leq j \leq n}
	= (\varepsilon_2, Q_1)
	= (\varepsilon_2, Q_2)
	= (\varepsilon_2, \partial_{y_j} Q)_{1 \leq j \leq n}
	= 0,
	\end{align*}
then
	\begin{align*}
	H(\varepsilon, \varepsilon)
	\geq \delta_0 \left( \int |\nabla \varepsilon|^2 dy + \int |\varepsilon|^2 e^{-|y|} dy \right).
	\end{align*}
\end{SP}

The proof of the Spectral Property in any dimension is not complete.
It has been proved in \cite{MerRa0} for dimension $n=1$ by using the explicit solution
$Q(x) = \left( \frac{3}{\cosh^2(2x)} \right)^\frac{1}{4}$ to  \eqref{eq:ground-state}. 
One can also find a computer assisted proof of the Spectral Property in dimensions $n = 2, 3, 4$ in \cite{FiMerRa}.
For dimension $5$ and higher see Remark 
\ref{re:N56_orthogonal}.

 The  Spectral Property is equivalent to the coercivity for $L_1$ and $L_2$ on quadratic forms,
 the study of which involving the ground
state solution $Q$ naturally appears in a perturbation setting when dealing with stability problem. 
These two operators are related to the Lyapounov functionals $L_\pm$,  where 
$L_+=-\De+1 -(1+\frac{4}{n})  Q^{\frac{4}{n}} $ and
$L_-=-\De+1 -Q^{\frac{4}{n}} $, see \cite{FiMerRa}. 


Using the Spectral Property, 
Merle and Rapha\"el obtained the following blowup rate for  \eqref{eq:nls-phi} in the absence of potentials,
see  \cite{MerRa1,FiMerRa,MeRaSz10}. 
\begin{theorem}\label{thm:log-log-phi}
Let $p = 1 + \frac{4}{n}$ in \eqref{eq:nls-phi}. 
Let $1\le n \leq 5$. 
 There exists a universal constant $\alpha^* > 0$ such that the following is true.
 Suppose	$\varphi_0 \in B_{\alpha^*}$ 
satisfies
	\begin{align*}
	\int |\nabla \varphi_0|^2 -\frac{n}{n+2} \int |\varphi_0|^{2+\frac{4}{n}}
	< 0.
	\end{align*}
Then  $\varphi \in C([0, T); H^1)$ is a blowup solution of  \eqref{eq:nls-phi} in finite time $T < \infty$,
which  admits the $\log$-$\log$ blowup rate
	\begin{align}\label{e:log-log_Q}
	\lim_{t \rightarrow T}
	\frac{\vert \nabla \varphi(t) \vert_2}{\vert \nabla Q \vert_2}
	\sqrt{ \frac{T - t}{\log \left| \log (T - t)  \right|} }
	= \frac{1}{\sqrt{2 \pi}}.
	\end{align}
\end{theorem}

\begin{remark}\label{re:N56_orthogonal} 
From \cite{FiMerRa} we know that the Spectral Property is true in dimensions $n=1,2,3,4$.
 If  $n= 5$, Theorem \ref{thm:log-log-phi} continues to hold true as soon as
  the Spectral Property verifies 
  the orthogonality condition 
  with  $(\veps_1, Q_1) = 0$ replaced by $(\veps_1, |y|^2Q) = 0$, which is numerically verified in \cite{FiMerRa}. 
For the above-mentioned reason, Theorem \ref{thm:log-log-phi} remains open in dimensions $n\ge 6$.
\end{remark}

\begin{remark} It is well-known that the $\log$-$\log$ law is {a} generic behavior for those blowup solutions  in the theorem,
whose proof relies on  algebraic cancellations  related to the topological degeneracy of the linear 
operators $L_1$ and  $L_2$ around $Q$. 
Such blowup rate is stable in the sense that
the set $\mathcal{U}_0$ in the log-log regime is open in $H^1$, where 
$\mathcal{U}_0$ denotes the set of all initial
data $\vphi_0$ in $B_{\al^*}$ 
 so that the flow  $\vphi_0\mapsto \vphi(t)$ of (\ref{eq:nls-phi})
collapses  in finite time $T^* <\iy$ with the $\log$-$\log$ speed given in (\ref{e:log-log_Q}), see \cite{Raph05,FiMerRa}.
\end{remark}

\section{The blowup rate 
for the rotational NLS}
In this section we  prove Theorem \ref{thm:log-log-u}.
We will always assume $p = 1 + \frac{4}{n}$ in both \eqref{eq:nls-u} and \eqref{eq:nls-phi}. 
We will need a  virial identity for \eqref{eq:nls-phi} and the $\mathcal{R}$-transform introduced in Proposition \ref{p:u-phi-R_pseudo-conformal}. 
This  transform  gives a relation between the two solutions of  \eqref{eq:nls-u} and \eqref{eq:nls-phi},
which is  coined as a combination of the lens transform and the rotation $e^{t A\cdot \nabla}$. One  can view it as certain pseudo-conformal symmetry in the rotational case, see \cite{Car,ZhuZhang} in the presence of a quadratic potential only, i.e., $M=0$ and $\ga\ne 0$. 

The following is a standard virial identity for \eqref{eq:nls-phi} in the weighted Sobolev space $\sH^1$, which can be proved by a direct calculation. 
\begin{lemma}\label{lem:virial-identity} 
Let $\varphi$ be a solution to the problem \eqref{eq:nls-phi} in $C([0, T), \sH^1)$. 
Define $J(t) := \int |x|^2 |\varphi|^2 dx$. Then
	\begin{align*}
	J'(t)
	= 4 \Im \int x \overline{\varphi} \cdot \nabla \varphi dx, \qquad
	J''(t)
	= 8 \mathcal{E}(\varphi_0),
	\end{align*}
where
	\begin{align*}
	\mathcal{E}(\varphi) = \int \left( |\nabla \varphi|^2 - \frac{n}{n+2} |\varphi|^{\frac{4}{n} + 2} \right) dx.
	\end{align*}
\end{lemma}



\begin{lemma}\label{lem:e^M}
\begin{enumerate}
\item[\textup{(a)}] Given any real $n$ by $n$ matrix $M$  and any  function $f$ in $C_0^\iy\cap L^2(\R^n)$, 
we have
	\begin{align}\label{E:etMx}
	e^{t (M x) \cdot \nabla} f(x)
	= f(e^{tM} x).
	\end{align}
\item[\textup{(b)}] If $M$ is a real skew-symmetric matrix, then $e^{tM}\in SO(n)$ for all $t$, where $SO(n)$ is the group of $n$ by $n$ orthogonal matrices with determinant $1$. 
Moreover, $(Mx) \cdot \nabla$ and $\Delta$ commute,
i.e.,
	\begin{align*}
	[(Mx) \cdot \nabla, \Delta] = 0.
	\end{align*}
\end{enumerate}\end{lemma} 
Part (a) can be proven by showing that both sides of \eqref{E:etMx} obey the ODE
\begin{align*}
	\partial_t F(t, x)= (Mx) \cdot \nabla F(t, x),  \quad F(0, x) = f(x).
	\end{align*}
Part (b) follows from a straightforward calculation.

\begin{proposition}\label{p:u-phi-R_pseudo-conformal}
Let $\varphi(t,x) \in C([0,T), H^1)$ be a solution to \eqref{eq:nls-phi} where $T > 0$.
Define the $\mathcal{R}$ transform $\vphi\mapsto \mathcal{R}(\vphi)$ to be
	\begin{align}\label{eq:phi-to-u}
	\mathcal{R}\vphi(t, x)
	:= \frac{1}{\cos^\frac{n}{2}(2 \gamma t)}
	e^{-i \frac{\gamma}{2} |x|^2 \tan(2 \gamma t)}
	\varphi\left( \frac{\tan(2 \gamma t)}{2 \gamma}, \frac{e^{tM} x}{\cos(2 \gamma t)} \right).
	\end{align}
Then $u=\mathcal{R}\vphi$ is a solution to \eqref{eq:nls-u} in $C([0, \frac{\arctan(2 \gamma T)}{2 \gamma}), \mathscr{H}^1)$.

Conversely, let $u(t,x) \in C([0,T^*), \mathscr{H}^1)$
be a solution to \eqref{eq:nls-u} where $T^* \in (0, \frac{\pi}{4 \gamma}]$.
Then $\vphi=\mathcal{R}\inv u$,  given by 
	\begin{align}\label{eq:u-to-phi}
	\varphi(t, x)
	:= \frac{1}{(1 + (2 \gamma t)^2)^\frac{n}{4}}
	e^{i \frac{\gamma^2 |x|^2 t}{1 + (2 \gamma t)^2}}
	u\left( \frac{\arctan(2 \gamma t)}{2 \gamma}, \frac{e^{-tM} x}{\sqrt{ 1 + (2 \gamma t)^2}} \right),
	\end{align}
is a solution to \eqref{eq:nls-phi} in $C([0, \frac{\tan(2 \gamma T^*)}{2 \gamma}), H^1)$, where $\mathcal{R}\inv $ is the inverse of $\mathcal{R}$.
\end{proposition}

\begin{proof} We will only briefly check \eqref{eq:phi-to-u} for $u=\mathcal{R}\phi$.
The other one $u\mapsto \vphi$ is the inverse transform. 
By direct computation,
we have
	\begin{align}\label{eq:u_t}
	\begin{split}
	u_t(t,x)
	& = n \gamma \frac{\sin(2 \gamma t)}{\cos^{\frac{n}{2}+1}(2 \gamma t)}
	e^{-i \frac{\gamma}{2} |x|^2 \tan(2 \gamma t)}
	\varphi\left( \frac{\tan(2 \gamma t)}{2 \gamma}, \frac{e^{tM} x}{\cos(2 \gamma t)} \right) \\
	& \quad - i \gamma^2 |x|^2 \frac{1}{\cos^{\frac{n}{2}+2}(2 \gamma t)}
	e^{-i \frac{\gamma}{2} |x|^2 \tan(2 \gamma t)}
	\varphi\left( \frac{\tan(2 \gamma t)}{2 \gamma}, \frac{e^{tM} x}{\cos(2 \gamma t)} \right) \\
	& \quad + \frac{1}{\cos^{\frac{n}{2}+2}(2 \gamma t)}
	e^{-i \frac{\gamma}{2} |x|^2 \tan(2 \gamma t)}
	\varphi_t\left( \frac{\tan(2 \gamma t)}{2 \gamma}, \frac{e^{tM} x}{\cos(2 \gamma t)} \right) \\
	& \quad + \frac{1}{\cos^{\frac{n}{2}+1}(2 \gamma t)}
	e^{-i \frac{\gamma}{2} |x|^2 \tan(2 \gamma t)}
	(e^{tM} M x) \cdot \nabla \varphi\left( \frac{\tan(2 \gamma t)}{2 \gamma}, \frac{e^{tM} x}{\cos(2 \gamma t)} \right) \\
	& \quad + 2 \gamma \frac{\sin(2 \gamma t)}{\cos^{\frac{n}{2}+2}(2 \gamma t)}
	e^{-i \frac{\gamma}{2} |x|^2 \tan(2 \gamma t)}
	(e^{tM} x) \cdot \nabla \varphi\left( \frac{\tan(2 \gamma t)}{2 \gamma}, \frac{e^{tM} x}{\cos(2 \gamma t)} \right).
	\end{split}
	\end{align}
To compute $\Delta u$ and $i (Mx) \cdot \nabla u$,
first note that
	\begin{align*}
	& \nabla\left( \varphi\left( \frac{\tan(2 \gamma t)}{2 \gamma}, \frac{e^{tM} x}{\cos(2 \gamma t)} \right) \right)
	= \frac{1}{\cos(2 \gamma t)} (e^{tM})^T \nabla \varphi\left( \frac{\tan(2 \gamma t)}{2 \gamma}, \frac{e^{tM} x}{\cos(2 \gamma t)} \right), \\
	& \Delta\left( \varphi\left( \frac{\tan(2 \gamma t)}{2 \gamma}, \frac{e^{tM} x}{\cos(2 \gamma t)} \right) \right)
	= \frac{1}{\cos^2(2 \gamma t)} \Delta \varphi\left( \frac{\tan(2 \gamma t)}{2 \gamma}, \frac{e^{tM} x}{\cos(2 \gamma t)} \right),
	\end{align*}
where we used 
\begin{enumerate} 
\item[(a)] If $B\in M_{n\times n}$ is a constant matrix, then 
$\nabla (\vphi(B x) )=B^T (\nabla \vphi)(Bx)$;
\item[(b)] $\De=\dive(\nabla)$, $\dive{{\bf F}}=\mathrm{tr}(\frac{\pa (F_1\cdots,F_n)}{\pa(x_1,\cdots,x_n)})=\text{trace of the Jacobian of}\;\mathbf{F} $;
\item[(c)] If $C$ is a constant square matrix, $\mathbf{W}=[w_1,\cdots,w_n]^T$ is a vector-valued function of $x\in\R^n$, then
\begin{align*}
&\dive(C \mathbf{W})=\trace([ \nabla w_1, \cdots, \nabla w_n] C^T).
\end{align*}
\item[(d)] $\trace(U^* \Lam U)=\trace(\Lam)$ if $U$ is a unitary matrix and $\Lam\in M_{n\times n}$.
\end{enumerate}
Thus  we obtain
	\begin{align}\label{eq:laplacian-u}
	\Delta u(t,x)
	& = -i n \gamma \frac{\sin(2 \gamma t)}{\cos^{\frac{n}{2}+1}(2 \gamma t)}
	e^{-i \frac{\gamma}{2} |x|^2 \tan(2 \gamma t)}
	\varphi\left( \frac{\tan(2 \gamma t)}{2 \gamma}, \frac{e^{tM} x}{\cos(2 \gamma t)} \right) \notag\\
	& \quad - \gamma^2 |x|^2 \frac{\sin^2(2 \gamma t)}{\cos^{\frac{n}{2}+2}(2 \gamma t)}
	e^{-i \frac{\gamma}{2} |x|^2 \tan(2 \gamma t)}
	\varphi\left( \frac{\tan(2 \gamma t)}{2 \gamma}, \frac{e^{tM} x}{\cos(2 \gamma t)} \right) \notag\\
	& \quad -i 2 \gamma \frac{\sin(2 \gamma t)}{\cos^{\frac{n}{2}+2}(2 \gamma t)}
	e^{-i \frac{\gamma}{2} |x|^2 \tan(2 \gamma t)}
	(e^{tM} x) \cdot \nabla\varphi\left( \frac{\tan(2 \gamma t)}{2 \gamma}, \frac{e^{tM} x}{\cos(2 \gamma t)} \right) \notag\\
	& \quad + \frac{1}{\cos^{\frac{n}{2}+2}(2 \gamma t)}
	e^{-i \frac{\gamma}{2} |x|^2 \tan(2 \gamma t)}
	\Delta \varphi\left( \frac{\tan(2 \gamma t)}{2 \gamma}, \frac{e^{tM} x}{\cos(2 \gamma t)} \right),
	\end{align}
and, noting that $(Mx) \cdot x = 0$, 
	\begin{align}\label{eq:imxdu}
	i (Mx) \cdot \nabla u
	 =  \frac{i}{\cos^{\frac{n}{2}+1}(2 \gamma t)}
	e^{-i \frac{\gamma}{2} |x|^2 \tan(2 \gamma t)}
	(M e^{tM} x) \cdot \nabla  \varphi\left( \frac{\tan(2 \gamma t)}{2 \gamma}, \frac{e^{tM} x}{\cos(2 \gamma t)} \right).
	\end{align}
Bring \eqref{eq:phi-to-u},
\eqref{eq:u_t},
\eqref{eq:laplacian-u},
and \eqref{eq:imxdu} into Cauchy problem \eqref{eq:nls-u},
and recall that $\varphi$ satisfies \eqref{eq:nls-phi} with $p=1+4/n$,
hence, we see that $u$ is a solution to \eqref{eq:nls-u} with $u(0,x)=\vphi(0,x)$. 

The above virtually shows that under the relation $u=\mathcal{R}\vphi\iff \vphi=\mathcal{R}\inv u$,
$u$ satisfies  \eqref{eq:nls-u} if and only if $\vphi$ satisfies  \eqref{eq:nls-phi}.  Therefore the second part of the proposition is also true.
\end{proof}
\begin{remark} The $\mathcal{R}$ transform also allows us to solve the linear equation for \eqref{eq:nls-u}.  
The equation 
$i\pa_t\vphi=-\De \vphi $ has the fundamental solution
	\begin{align*}
	e^{i t \Delta} (x,y)
	= \frac{1}{(4 \pi i t)^\frac{n}{2}} e^{i \frac{|x - y|^2}{4t}} .
	\end{align*}
Applying \eqref{eq:phi-to-u}  we then  obtain the fundamental solution to $i\pa_t u=H_{A, V}u$: 
	\begin{align}\label{e:fund-U(t)}
	 e^{-itH_{A,V}}(x,y)=\left( \frac{\gamma}{2 \pi i \sin(2 \gamma t)} \right)^\frac{n}{2}
	e^{i \frac{\gamma}{2} (|x|^2+|y|^2) \cot(2 \gamma t)}
	e^{-i \gamma \frac{(e^{tM} x) \cdot y}{\sin(2 \gamma t)}}.	
	\end{align}
This expression 
 is significantly simpler than the one  in \cite{ChMarS16} per Mehler's formula. 
  To our best knowledge,  \eqref{e:fund-U(t)} might be  the first simply unified explicit formula compared with 
  \cite{Kita80,HaoHsiaoLi1, HaoHsiaoLi2, AnMaSpar}. 
\end{remark}

Now we are ready to prove Theorem \ref{thm:log-log-u}.
\begin{proof}[Proof of Theorem \ref{thm:log-log-u}]
Let $u \in C([0, T), \mathscr{H}^1)$ be the blowup solution to the  problem \eqref{eq:nls-u},
where $[0, T)$ is the maximal interval of existence.
Then by \eqref{eq:u-to-phi},
there is a $\varphi(t,x) \in C([0, \frac{\tan(2 \gamma T)}{2 \gamma}), \mathscr{H}^1)$ that solves  \eqref{eq:nls-phi},
where $[0, \frac{\tan(2 \gamma T)}{2 \gamma})$ is the maximal interval of existence.
Note that $u_0= \varphi_0$, 
and according to Theorem \ref{thm:log-log-phi}, 
the condition  \eqref{eq:negative-initial-energy} suggests that
 $\varphi$ is a blowup solution.
Recall  from \eqref{eq:phi-to-u}, for $T\in (0,\frac{\pi}{4\ga}]$
	\begin{align*}
	u(t, x)
	= \frac{1}{\cos^\frac{n}{2}(2 \gamma t)}
	e^{-i \frac{\gamma}{2} |x|^2 \tan(2 \gamma t)}
	\varphi\left( \frac{\tan(2 \gamma t)}{2 \gamma}, \frac{e^{tM} x}{\cos(2 \gamma t)} \right).
	\end{align*}
Then for all $t \in [0, T)$,
we have
	\begin{align}\label{eq:Du-I1-I2}
	\begin{split}
	\nabla_x u(t,x)
	& = -i \gamma x \frac{\sin(2 \gamma t)}{\cos^{\frac{n}{2}+1}(2 \gamma t)}
	e^{-i \frac{\gamma}{2} |x|^2 \tan(2 \gamma t)}
	\varphi\left( \frac{\tan(2 \gamma t)}{2 \gamma}, \frac{e^{tM} x}{\cos(2 \gamma t)} \right) \\
	& \quad + \frac{1}{\cos^{\frac{n}{2}+1}(2 \gamma t)}
	e^{-i \frac{\gamma}{2} |x|^2 \tan(2 \gamma t)}
	(e^{tM})^T \nabla \varphi\left( \frac{\tan(2 \gamma t)}{2 \gamma}, \frac{e^{tM} x}{\cos(2 \gamma t)} \right) \\
	& := I_1 + I_2.
	\end{split}
	\end{align}

For $ I_1 $, a change of variable gives
	\begin{align*}
	\vert I_1 \vert_{2}
	= \gamma \sin(2 \gamma t) \left\vert x \varphi\left( \frac{\tan(2 \gamma t)}{2 \gamma}, x \right) \right\vert_{2}.
	\end{align*}
Let $J(t) := \vert x \varphi(t, x) \vert_{2}^2$.
Then \begin{align*}
	J(t)
	= J(0) + J'(0) t + \int_0^t J''(\tau) \,(t - \tau) d\tau.
	\end{align*}
Note that
	\begin{align*}
	|J(0)|
	= \vert x \varphi_0 \vert_2^2
	\leq \vert \varphi_0 \vert_{\mathscr{H}^1}^2,
	\end{align*}
and by Lemma \ref{lem:virial-identity},
we have
\begin{align*}
	|J'(0)|
	& = \left| 4 \Im \int x \overline{\varphi_0} \cdot \nabla \varphi_0 dx \right|
	\leq 4 \| x \varphi_0 \|_2 \| \nabla \varphi_0 \|_2
	\leq  \| \varphi_0 \|_{\mathscr{H}^1}^2, \\
	J''(t)
	& = 8 \mathcal{E}(\varphi_0).
	\end{align*}
Thus
	\begin{align*}
	& \left\vert x \varphi\left( \frac{\tan(2 \gamma t)}{2 \gamma}, x \right) \right\vert_{2}^2
	= J \left( \frac{\tan(2 \gamma t)}{2 \gamma} \right) \\
	& \leq |J(0)|
	+ |J'(0)| \frac{\tan(2 \gamma t)}{2 \gamma}
	+ 4 \mathcal{E}(\varphi_0) \left( \frac{\tan(2 \gamma t)}{2 \gamma} \right)^2 \\
	& \leq \vert \varphi_0 \vert_{\mathscr{H}^1}^2
	+  \vert\varphi_0 \vert_{\mathscr{H}^1}^2 \frac{\tan(2 \gamma T)}{2 \gamma}
	+ 4 \mathcal{E}(\varphi_0) \left( \frac{\tan(2 \gamma T)}{2 \gamma} \right)^2,
	\end{align*}
and so
	\begin{align}\label{eq:I1-estimate}
	\vert I_1 \vert_{2} \leq C(\varphi_0, T).
	\end{align}
	
For $I_2 $, in view of Lemma \ref{lem:e^M}, $e^{tM^T}\in SO(n)$,  
a change of variable gives for $t\in [0,T)$, ($T\le\pi/4\ga$)
	\begin{align*}
	\vert I_2 \vert_{2}
	= \frac{1}{\cos (2 \gamma t)} \left\vert \nabla \varphi\left( \frac{\tan(2 \gamma t)}{2 \gamma}, x \right) \right\vert_{2}.
	\end{align*}
As $t \rightarrow T$,
$\frac{\tan(2 \gamma t)}{2 \gamma} \rightarrow \frac{\tan(2 \gamma T)}{2 \gamma}$,
so by Theorem \ref{thm:log-log-phi},
	\begin{align*}
	\lim_{t \rightarrow T}
	\frac{\left\vert \nabla \varphi \left( \frac{\tan(2 \gamma t)}{2 \gamma}, x \right) \right\vert_{2}}{\vert\nabla Q \vert_{2}}
	\sqrt{ \frac{\frac{\tan(2 \gamma T)}{2 \gamma} - \frac{\tan(2 \gamma t)}{2 \gamma}}
	{\log \left| \log \left( \frac{\tan(2 \gamma T)}{2 \gamma} - \frac{\tan(2 \gamma t)}{2 \gamma} \right) \right|} }
	= \frac{1}{\sqrt{2 \pi}}.
	\end{align*}
Note that as $t \rightarrow T$,
there are
	\begin{align*}
	\frac{\sin(2 \gamma (T - t))}{2 \gamma (T - t)} \rightarrow 1 \qquad \text{and} \qquad
	\frac{\log \sin(2 \gamma (T - t))}{\log(T - t)} \rightarrow 1,
	\end{align*}
so the above blowup rate can be simplified as
	\begin{align*}
	\lim_{t \rightarrow T}
	\frac{\left\vert \nabla \varphi \left( \frac{\tan(2 \gamma t)}{2 \gamma}, x \right) \right\vert_{2}}{\vert \nabla Q \vert_{2}}
	\sqrt{ \frac{T - t}{\log \left| \log (T - t) \right|} }
	= \frac{\cos(2 \gamma T)}{\sqrt{2 \pi}}
	\end{align*}
to yield 
	\begin{align}\label{eq:I2-estimate}
	\lim_{t \rightarrow T}
	\frac{\left\vert I_2 \right\vert_{2}}{\vert \nabla Q\vert_2}
	\sqrt{ \frac{T - t}{\log \left| \log (T - t) \right|} }
	= \frac{1}{\sqrt{2 \pi}}.
	\end{align}
Therefore, combining \eqref{eq:I1-estimate} and \eqref{eq:I2-estimate} we obtain
	\begin{align*}
	\lim_{t \rightarrow T}
	\frac{\left\vert \nabla u \right\vert_{2}}{\vert \nabla Q \vert_{2}}
	\sqrt{ \frac{T - t}{\log \left| \log (T - t) \right|} }
	= \frac{1}{\sqrt{2 \pi}}.
	\end{align*}
\end{proof}

\subsection{Blowup rate at the ground state $Q$}  We conclude with some discussions on the wave collapse rates for \eqref{eq:nls-u}
when the initial data is near $Q$, which could be a subtle issue.  
Notice that this $Q$ is not the ground state for \eqref{eq:nls-u}, instead, it is the one for \eqref{eq:nls-phi}.
 If $\vert u_0 \vert_{2} = \vert Q \vert_{2}$, 
then the wave collapse for \eqref{eq:nls-u}  is different than in the case  $\norm{u_0}_2 >\norm{Q}_2$.  
Applying the transform \eqref{eq:phi-to-u} to the solitary wave $\vphi=e^{it}Q$ 
we can construct a blowup solution with blowup rate $(T-t)\inv$:
	\begin{align}\label{eU:lens-R}
	u(t, x)
	= \frac{1}{\cos^\frac{n}{2}(2 \gamma t)}
	e^{-i \frac{\gamma}{2} |x|^2 \tan(2 \gamma t)}
	e^{i \frac{\tan(2\gamma t)}{2 \gamma}}
	Q\left( \frac{ e^{tM} x}{\cos(2\gamma t)} \right).
	\end{align}
 One easily checks that $u$ blows up at $T=\frac{\pi}{4 \gamma}$ satisfying
	\begin{align*}
	\vert \nabla u \vert_{2}
	\approx \frac{1}{ (\frac{\pi}{2} - 2\ga t)} \vert\nabla Q \vert_{2} \qquad \text{as} \ t \rightarrow T=\frac{\pi}{4 \gamma}.
	\end{align*}  
Note that solutions of the form  \eqref{eU:lens-R} with such blowup time and singularity can also be obtained with  other nonpositive, non-radial bound states $Q_b$ as the profile in place of $Q$,  
where $\norm{Q_b}_2>\norm{Q}_2$ and $n\ge 2$. 
Suppose that $u$ is a blowup solution to \eqref{eq:nls-u} with blowup time $T^* < \frac{\pi}{4 \gamma}$ and $\vert u_0 \vert_{2} = \vert Q \vert_{2}$.
Then by means of \eqref{eq:u-to-phi} we may define a blowup solution  to \eqref{eq:nls-phi} with the same initial data which  blows up at  $T_0 = \frac{\tan(2\gamma T)}{2 \gamma}$.
Merle \cite{Mer} showed that up to the scaling and phase invariances of \eqref{eq:nls-phi},
the only minimal mass blowup solutions are of the form
	\begin{align*}
	\vphi(t,x)=\frac{1}{(T_0 - t)^\frac{n}{2}} e^{-\frac{i |x|^2}{4 (T_0 - t)}} e^\frac{i}{T_0 - t} Q\left( \frac{x}{T_0 - t} - x_0 \right)
	\end{align*}
for some $x_0 \in \mathbb{R}^n$. 
By the $\mathcal{R}$-transform \eqref{eq:phi-to-u} and the uniqueness of \eqref{eq:nls-u} (Proposition \ref{p:blup-altern-conserv}),  
we then establish a characterization for all minimal mass blowup solutions of  \eqref{eq:nls-u}. 
 \begin{proposition}\label{c:charact-mini-pc} Let $\vert u_0 \vert_{2} = \vert Q \vert_{2}$. Let 
 $u$ be a blowup solution of \eqref{eq:nls-u} on $[0,T)$ with $T:=T^* < \frac{\pi}{4 \gamma}$. Then $u$ must assume the following form
 (up to scaling and phase invariance): There exists $x_1\in\R^n$ such that 
\begin{equation}
\begin{aligned}\label{eU:pc-M}
u(t,x)=& \left(\frac{ 2 \gamma \cos(2\gamma T) }{\sin (2 \gamma (T - t))}  \right)^{\frac{n}{2}}
	e^{-i \frac{\gamma}{2} |x|^2 \cot(2 \gamma (T - t))}
	e^{i 2 \gamma \frac{\cos(2 \gamma T) \cos(2 \gamma t)}{\sin(2 \gamma (T - t))}}\\
	&\times Q\left( \frac{2 \gamma \cos(2 \gamma T)  e^{tM}x}{\sin(2 \gamma (T - t))} - x_1 \right).
	\end{aligned}
\end{equation}
Moreover, $\norm{\nabla u}_2\approx (T-t)\inv$ as $t\to T$. 
\end{proposition}   
	Note that the blowup solution given in \eqref{eU:lens-R}  is not covered by  \eqref{eU:pc-M}.
	For example, if $T^*=\pi/8\ga$, $x_1=0$, then \mbox{$u_0(x)=(2\ga)^{n/2} e^{-i\frac{\ga}{2}|x|^2} e^{i2\ga} Q(2\ga x)$.}
Rather,  \eqref{eU:lens-R} can be viewed as a bordering case 
of the assertion in Proposition \ref{c:charact-mini-pc} corresponding to $T^*=\pi/4\ga$. 

\begin{remark}
If the initial value is of the form $u_0 = (1 + \varepsilon) Q$ with $0 < \varepsilon <\sqrt{1+ \frac{\alpha^*}{\| Q \|_2^2}}-1$,
then the corresponding solution to \eqref{eq:nls-u} will blowup at the rate stated in Theorem \ref{thm:log-log-u}.
Indeed,  by the range for $\varepsilon$ and 
the Pohozaev identity $\int |\nabla Q|^2 =\frac{n}{n+2} \int |Q|^{2+\frac{4}{n}} $, 
it is easy to verify  $u_0\in B_{\al^*}$ and condition \eqref{eq:negative-initial-energy}.
\end{remark}

\begin{remark}  
For large initial data one can also derive a general lower bound for the collapse rate. 
If the solution of \eqref{eq:nls-u} satisfies $\lim_{t \rightarrow T^*} \vert\nabla u \vert_{2} = \infty$, 
then there exists  $C=C_{p,n}> 0$ such that
	\begin{align*}
	\vert\nabla u(t) \vert_{2}
	\geq C (T^* - t)^{-(\frac{1}{p-1} - \frac{n-2}{4})}.
	\end{align*}
	 This follows from quite standard argument as in \cite{CazW90} that is used to show the l.w.p 
  and blowup alternative for  the Cauchy problem \eqref{eq:nls-u} on $[0,T^*)$.  
In the $L^2$-critical case $p = 1 + \frac{4}{n}$,  the lower bound becomes
	\begin{align*}
	\vert \nabla u(t) \vert_{2}\geq C (T^* - t)^{-\frac{1}{2}}.
	\end{align*}
\end{remark}


\vspace{.1520in}
\nd{\bf Acknowledgment} The authors would like to thank Remi Carles  and Chenjie Fan for 
 helpful comments and communications.

\end{document}